%% file: main.tex
\documentclass[letterpaper,12pt]{article}
\usepackage{natbib} \bibpunct{(}{)}{;}{a}{}{,}
\usepackage{fullpage}
\usepackage{amsmath,amssymb,amsthm}
\usepackage{enumerate}
\usepackage{appendix}
\usepackage{graphicx}
\usepackage{color}
\usepackage{hyperref}
\usepackage{algorithm}
\usepackage{algpseudocode}

\hypersetup{
        colorlinks   = true,
        linkcolor    = blue,
        citecolor    = green
}
\usepackage{authblk}

\input{newcommands}

\begin{document}

\title{Fast Generation of Exchangeable Sequence of Clusters Data}
\author{Keith Levin}
\affil{\small University of Wisconsin-Madison}
\author{Brenda Betancourt}
\affil{\small NORC at the University of Chicago}

\maketitle

\begin{abstract}
Recent advances in Bayesian models for random partitions have led to the formulation and exploration of Exchangeable Sequences of Clusters (ESC) models.
Under ESC models, it is the cluster sizes that are exchangeable, rather than the observations themselves.
This property is particularly useful for obtaining microclustering behavior, whereby cluster sizes grow sublinearly in the number of observations, as is common in applications such as record linkage, sparse networks and genomics.
Unfortunately, the exchangeable clusters property comes at the cost of projectivity.
As a consequence, in contrast to more traditional Dirichlet Process or Pitman-Yor process mixture models, samples a priori from ESC models cannot be easily obtained in a sequential fashion and instead require the use of rejection or importance sampling.
In this work, drawing on connections between ESC models and discrete renewal theory, we obtain closed-form expressions for certain ESC models and develop faster methods for generating samples a priori from these models compared with the existing state of the art.
In the process, we establish analytical expressions for the distribution of the number of clusters under ESC models, which was unknown prior to this work.
\end{abstract}

\section{Introduction} \label{sec:intro}
\input{intro.tex}

\section{Main Results} \label{sec:results}
\input{theory.tex}

\section{Experiments} \label{sec:expts}
\input{experiments.tex}

\section{Discussion and Conclusion} \label{sec:conclusion}

We have addressed two outstanding issues in ESC models: the behavior of the number of clusters $K_n$ and the matter of sampling a priori from these models.
A number of natural follow-up questions present themselves.
For example, all known results concerning the microclustering property in ESC models require that the cluster size distribution $\bmu$ have finite expectation.
It is natural to ask whether the microclustering property continues to hold if $\bmu$ has infinite expectation, and how the size of the largest cluster grows in such situations.

One possible criticism of Algorithm~\ref{alg:main} is that it requires $O(n^2)$ up-front runtime to compute the probabilities $\Pr[ X_k = s; m ]$ for all $1 \le s \le m$.
Absent particular structure in the cluster size distribution $\bmu$, it requires a new $O(n^2)$ runtime computation any time $\bmu$ is updated.
We stress that Algorithm~\ref{alg:main} is not aimed at this situation, but rather is meant for faster a priori sampling, such as in the context of prior calibration.
Nonetheless, future work should investigate speeding up the evaluation of these probabilities for use in Algorithm~\ref{alg:main}, perhaps using approximation techniques similar to those deployed in \cite{BysArbKinDes2020}.

\bibliographystyle{plainnat}
\bibliography{biblio}

\newpage
\onecolumn
\appendix

\section{Proof of Theorem~\ref{thm:algmain:correctness}}
\label{apx:alg}
\input{apx_algmain.tex}

\section{Selected Cluster Size Distributions}
\label{apx:examples}
\input{apx_examples.tex}

\end{document}

%% file: newcommands.tex
\newtheorem{theorem}{Theorem}

\newtheorem{lemma}{Lemma}

\renewcommand{\Pr}{\mathbb{P}}

\newcommand{\E}{\mathbb{E}}
\newcommand{\R}{\mathbb{R}}

\newcommand{\Bhat}{\hat{B}}

\newcommand{\bmu}{\mathbf{\mu}}

\newcommand{\Pois}{\operatorname{Pois}}

\newcommand{\iid}{\stackrel{\text{i.i.d.}}{\sim}}

%% file: intro.tex
Random partitions are integral to a variety of Bayesian clustering methods, with applications in
text analysis~\citep{BleNgJor2003,Blei2012}
genetics~\citep{PriSteDon2000,FalStePri2003},
entity resolution~\citep{BinSte2022}
and
community detection~\citep{LegRigDurDun2020}, to name but a few.
Most prominent among random partition models are those based on Dirichlet processes and Pitman-Yor processes \citep{Antoniak1974,sethuraman94constructive,ishwaran03generalized}, including the famed Chinese Restaurant Process (CRP).
One major drawback of these models is that they generate partitions in which one or more cells of the partition grows linearly in the number of observations $n$.
This property is undesirable in applications to, for example, record linkage and social network modeling, where data commonly exhibits a large number of small clusters.
For these applications, a different mechanism is needed that better captures the growth of cluster sizes with $n$.

The solution to this issue is to deploy models with the {\em microclustering} property, whereby the size of the largest cluster grows sublinearly in the number of observations $n$.
An early attempt to develop microclustering models appeared in \cite{ZanellaETALneurips2016}.
The authors were motivated by record linkage applications~\citep{BinSte2022} where clusters are expected to remain small even as the number of observations increases.
This initial class of models, constructed under the Kolchin representation of Gibbs partitions~\citep{kolchin1971problem}, places a prior $\kappa$ on the number of clusters $K$, and then draws from a distribution $\mu$ over cluster sizes conditional on $K$.
This approach is comparatively simple, admitting an algorithm that facilitates sampling a priori and a posteriori similar to the Chinese Restaurant Process \citep[CRP;][]{aldous1985}.
Unfortunately, the distributions of the number of clusters and the size of a randomly chosen cluster are not straightforwardly related to the priors $\kappa$ and $\mu$.
More to the point, it is not yet theoretically proven that this family of models indeed exhibits the microclustering property.

More recently, \cite{BetZanSte2020} considered a different approach to the microclustering property, called {\em Exchangeable Sequences of Clusters} (ESC) models.
These models belong to the class of finitely exchangeable Gibbs partitions~\citep{GnePit2006,Pitman2006}, named for the fact that the cluster sizes $S_1,S_2,\dots$ are finitely exchangeable.
An ESC model is specified by a distribution $\mu$ over cluster sizes (or a prior over such distributions), and a partition is generated by drawing cluster sizes independently from $\mu$ conditional on the event that these cluster sizes sum to $n$.
That is, having specified a distribution $\mu$ on the positive integers, we draw cluster sizes $S_1,S_2,\dots$ i.i.d.\ $\mu$, conditional on the event
\begin{equation} \label{eq:def:En}
E_n = \left\{ \exists K : \sum_{j=1}^K S_j = n \right\}.
\end{equation}
The advantage of this model is that the prior $\mu$ straightforwardly encodes a distribution over cluster sizes, in the sense that the size of a randomly chosen cluster is (in the large-$n$ limit) distributed according to $\mu$ \citep[][Theorem 2]{BetZanSte2020}.
Furthermore, unlike the model proposed in \cite{ZanellaETALneurips2016}, the microclustering property has been theoretically established for ESC models \citep[][Theorem 3]{BetZanSte2020}.

While ESC models are more interpretable and have better-developed theory than previously-proposed microclustering models, there is no known relationship between the cluster size distribution $\mu$ and the number of clusters $K$ under these models.
Recently, \cite{NatIorHeiMayGle2021} (Proposition 2) established the distribution of the number of clusters $K$ for the case where $\mu$ is a shifted negative binomial, one of the specific models first proposed by \cite{BetZanSte2020}. 
\cite{BysArbKinDes2020} established the behavior of $K$ under a related class of Gibbs-type processes.
Nonetheless, a general description of the behavior of $K$ under ESC models remains open.
Additionally, since ESC models require conditioning on $E_n$, previous approaches to sampling a priori amount to drawing repeatedly from $\mu$ and checking whether or not the cluster sizes $S_1,S_2,\dots$ satisfy the condition in event $E_n$.
In this paper, we resolve both of these issues by
\begin{enumerate}
\item Establishing analytic expressions for the distribution of the number of clusters under ESC models by relating the ESC generative process to known results in renewal theory and enumerative combinatorics.
\item Leveraging these connections with enumerative combinatorics to more efficiently sample from ESC models.
\end{enumerate}

\subsection{Related work on ESC models}
Apart from the prior works outlined above, the literature on the microclustering property is scarce but diverse.
Previous work includes models that sacrifice finite exchangeability to handle data with a temporal component~\citep[e.g., arrival times][]{DiBenedetto2017}, general finite mixture models with constraints on cluster sizes~\citep{klami2016probabilistic, jitta2018controlling, silverman2017bayesian}, and models for sparse networks based on random partitions with power-law distributed cluster sizes~\citep{Bloem-Reddy18}.
Recently, \cite{LeeSan2022} considered the question of {\em balance} in cluster sizes, as encoded by majorization of cluster size vectors. 
Clearly, this is an emergent area of research with a variety of applications for which efficient sampling alternatives are crucial.



%% file: theory.tex
We begin by defining the ESC model more rigorously.
Our goal is to generate a partition of $[n] = \{1,2,\dots,n\}$.
Under the ESC model, this is done by first selecting a distribution $\bmu$ on the positive integers according to a prior $P_\bmu$.
Having picked such a distribution $\bmu$, the ESC model generates partition sizes by drawing $S_1,S_2,\dots$ i.i.d.\ from $\bmu$, conditional on the event $E_n$ defined in Equation~\eqref{eq:def:En}, according to the following procedure:
\begin{enumerate}
\item Draw $\bmu \sim P_{\bmu}$
\item Draw $S_1,S_2,\dots \iid \bmu$ conditional on the event $E_n$.
\item Define $K_n$ to be the unique integer
        such that $\sum_{j=1}^{K_n} S_j = n$.
\item Assign the $n$ observations to $K_n$ clusters by randomly permuting the vector
        \begin{equation*}
        (1,1,\dots,1,2,2,\dots,2,\dots,K_n,K_n,\dots,K_n)
        \end{equation*}
in which $1$ appears $S_1$ times, $2$ appears $S_2$ times, etc.
\end{enumerate}

As discussed in the introduction, this model raises two key challenges.
First, while $\bmu$ naturally encodes the (asymptotic) cluster size distribution, it is not immediately clear how to relate the behavior of the number of clusters $K_n$ to $\bmu$ or to our prior $P_{\bmu}$.
This raises a challenge for the purposes of interpretability and usability of the model.
Second, generating samples a priori from this distribution is non-trivial, since one must condition on the event $E_n$ that $\sum_{j=1}^{K_n} S_j = n$.
We address both of these concerns by drawing on the connections between the ESC model, renewal theory and enumerative combinatorics.

\subsection{Generating ESC Clusterings}
\label{subsec:theory:gen}

Let us consider the matter of generating clusterings from ESC models.
\cite{BetZanSte2020} suggest drawing $S_1,S_2,\dots$ i.i.d.\ according to $\bmu$ until $\sum_{j=1}^k S_j \ge n$ for some $k \le n$.
If equality holds, then $(S_1,S_2,\dots,S_k)$ is a valid sequence of cluster sizes (i.e., the event $E_n$ holds), otherwise a new sequence is generated.
Unfortunately, on average, this procedure must be repeated $1/\Pr[E_n\mid \bmu]$ times before a valid sequence is generated.
Thus, crucial to this approach is that $\Pr[E_n \mid \bmu]$ be bounded away from zero for large $n$.
This fact is established in \cite{BetZanSte2020} for the case where $\bmu$ has finite mean by identifying the cluster sizes $S_1,S_2,\dots$ with the waiting times of a discrete renewal process and appealing to the following result \citep[see, for example, Theorem 2.6 in][]{renewal}.
\begin{lemma} \label{lem:renewal}
Let $\bmu$ be a distribution on the positive integers with finite mean and generate $S_1,S_2,\dots \iid \bmu$.
With $E_n$ as defined in Equation~\eqref{eq:def:En},
\begin{equation*}
\lim_{n \rightarrow \infty} \Pr[ E_n \mid \bmu ] = \frac{1}{\E[ S_1 \mid \bmu]}.
\end{equation*}
\end{lemma}

Trouble arises in the event that $\E[ S_1 \mid \bmu ]$ is large (or infinite), since then we may need to generate many samples $S_1,S_2,\dots$ from $\bmu$ before the sampler generates a usable sequence.
To alleviate this issue and allow for the possibility that $\bmu$ has infinite expectation, we propose an alternative approach to generating cluster sizes conditional on $E_n$.
We begin by writing, for positive integers $s_1,s_2,\dots$,
\begin{equation} \label{eq:condprob}
\begin{aligned}
\Pr& \left[ S_1=s_1,S_2=s_2,\dots \mid E_n, \bmu \right] \\
~~~~~~&= \Pr\left[ S_1=s_1 \mid E_n, \bmu \right] 
  \Pr\left[ S_2=s_2,S_3=s_3,\dots \mid S_1=s_1, E_n, \bmu \right].
\end{aligned}
\end{equation}
Since the variables $S_1,S_2,\dots$ are drawn i.i.d., we have
\begin{equation} \label{eq:seqshift}
\Pr\left[ S_2=s_2,S_3=s_3,\dots \mid S_1=s_1, E_n, \bmu \right]
= \Pr\left[ S_1=s_2,S_2=s_3,\dots \mid E_{n-s_1}, \bmu \right].
\end{equation}
Similarly,
\begin{equation*}
\Pr\left[ E_n \mid S_1=s_1, \bmu \right]
= \Pr\left[ \exists k : \sum_{j=2}^k S_j = n-s_1 ~\Big|~ \bmu \right]
= \Pr\left[ E_{n-s_1} \mid \bmu \right],
\end{equation*}
from which we have
\begin{equation*}
\Pr\left[ S_1=s_1 \mid E_n, \bmu \right]
= \frac{ \Pr\left[ E_n \mid S_1=s_1, \bmu \right] \Pr[ S_1 = s_1 \mid \bmu ] }
        { \Pr\left[ E_n \mid \bmu \right] }
= \frac{ \Pr\left[ E_{n-s_1} \mid \bmu \right] \bmu_{s_1} }
        { \Pr\left[ E_n \mid \bmu \right] }.
\end{equation*}
Plugging this and Equation~\eqref{eq:seqshift} into Equation~\eqref{eq:condprob}
, we have, for $1 \le s_1 \le n$,
\begin{equation} \label{eq:seqstep}
\Pr\left[ S_1=s_1,S_2=s_2,\dots \mid E_n, \bmu \right]
= \frac{
        \Pr\left[ S_1=s_2,S_2=s_3,\dots \mid E_{n-s_1}, \bmu \right]
        \Pr\left[ E_{n-s_1} \mid \bmu \right]
        \bmu_{s_1} 
        }
        { \Pr[ E_n \mid \bmu ] }.
\end{equation}
This equation suggests a recursive approach to generating cluster size sequences, which we formalize in Algorithm~\ref{alg:main}.
Crucially, we note that this algorithm avoids the runtime dependence on $\Pr[E_n\mid \bmu ]$ exhibited by the na\"{i}ve rejection sampling approach.

\begin{algorithm} \label{alg:main}
\caption{Given distribution $\bmu = ( \mu_n )_{n=1}^\infty$, generate $S_1,S_2,\dots \mid E_n$.} \label{alg:genpart}
\begin{algorithmic}[1]
\State Compute the sequence $\Pr[E_t \mid \bmu]$ for $t=1,2,\dots,n$.
\State $m \gets n;~~~k \gets 1$
\While{$m > 0$}
  \State Draw $X_k$ according to $\Pr[ X_k = s; m ] = \bmu_{s} \Pr[E_{m-s} \mid \bmu] / \Pr[ E_m \mid \bmu ]$ for $s \in \{1,2,\dots,m \}$
  \State $m \gets m-X_k;~~~k \gets k+1$
\EndWhile

\State Return $(X_1,X_2,\dots,X_{k-1})$
\end{algorithmic}
\end{algorithm}

\begin{theorem} \label{thm:algmain:correctness}
For any $s_1,s_2,\dots \in [n]$ satisfying $\sum_{j=1}^k s_j = n$, the sequence $(X_1,X_2,\dots,X_k)$ generated by Algorithm~\ref{alg:genpart} satisfies
\begin{equation*}
\Pr\left[ X_1=s_1,X_2=s_2,\dots,X_k=s_k \mid \bmu \right]
=
\Pr\left[ S_1=s_1,S_2=s_2,\dots,S_k=s_k \mid E_n, \bmu \right]
\end{equation*}
\end{theorem}
\begin{proof}
This follows from the construction of Algorithm~\ref{alg:main} and repeated application of Equation~\eqref{eq:seqstep}.
A detailed proof can be found in Appendix~\ref{apx:alg}.
\end{proof}

Algorithm~\ref{alg:main} generates samples from an ESC model without the rejection sampling approach initially proposed in \cite{BetZanSte2020}, provided that we can compute $ \Pr[ E_n \mid \bmu ]$ for arbitary choices of $n \ge 0$.
Viewing the cluster sizes $S_1,S_2,\dots$ as the waiting times of a discrete-time renewal process~\citep{renewal}, $E_n$ corresponds to the event that a renewal occurs at time $n$.
A key result from renewal theory relates $\Pr[ E_n \mid \bmu ]$ and the cluster size distribution $\bmu$ via their generating functions.
Let $M(s)$ denote the ordinary moment generating function of $\bmu = ( \mu_n )_{n=1}^\infty$.
That is, for $s \ge 0$,
\begin{equation*} 
M(s) = \sum_{k=0}^\infty \mu_k s^k,
\end{equation*}
where $\mu_0 = 0$ by assumption (i.e., in the language of the ESC model, there are no empty clusters; in the language of renewal theory, waiting times are positive).
For each $n=0,1,2,\dots$, let $u_n = \Pr[ E_n ]$, with $u_0 = 1$ by convention (i.e., a renewal always occurs at time $0$).
Letting $U(s)$ be the generating function of the sequence $(u_n)_{n=0}^\infty$, one can show \citep[see, e.g.,][Proposition 2.1]{renewal}
that for all $s \ge 0$,
\begin{equation} \label{eq:renewal:OGF}
U(s) = \sum_{k=0}^\infty u_k s^k  = \frac{ 1 }{ 1 - M(s) }.
\end{equation}
This suggests a natural approach to computing $\Pr[E_n \mid \bmu ] = u_n$ using the fact that $u_n$ can be determined from the $n$-th derivative of $U(s)$ evaluated at $s=0$.
Defining the functions $f(z) = 1/z$ and $g(s) = 1-M(s)$, observe that for all $n = 0,1,2,\dots$,
\begin{equation*}
U^{(n)}(s) = \frac{ d^n }{ d s^n } f( g(s) ) 
\end{equation*}
and for $n=1,2,\dots$, we have
\begin{equation*}
f^{(n)}(z) = \frac{ (-1)^{n} n! }{ z^{n+1} },
~~~\text{ and }~~~
g^{(n)}(s) = -M^{(n)}(s).
\end{equation*}
Applying Fa\'{a} di Bruno's formula \citep[][Theorem 11.4]{Charalambides2002},
\begin{equation*} \begin{aligned}
U^{(n)}(s)
&= 
\frac{ d^n }{ d s^n } f( g(s) )
=
\sum_{k=1}^n f^{(k)}(g(s)) B_{n,k}\left(g^\prime(s),g^{\prime\prime}(s),
                                \dots,g^{(n-k+1)}(s) \right) \\
&= \sum_{k=1}^n \frac{ (-1)^k k! }{ (1- M(s))^{k+1} }
        B_{n,k}\left( -M^{\prime}(s), -M^{\prime \prime}(s),
                \dots, - M^{(n-k+1)}(s) \right),
\end{aligned} \end{equation*}
where $B_{n,k}$ is the $k$-th partial exponential Bell polynomial
\citep{Charalambides2002},
\begin{equation} \label{eq:def:ExponentialBell}
B_{n,k}(x_1,x_2,\dots,x_{n-k+1})
= \sum_{j_1,j_2,\dots,j_{n-k+1}}
	\frac{ n! }{j_1! j_2! \cdots j_{n-k+1}! }
        \prod_{i=1}^{n-k+1} \left( \frac{ x_i }{ i! } \right)^{j_i},
\end{equation}
where the sum is over all nonnegative integers $j_1,j_2,\dots,j_{n-k+1}$ satisfying $\sum_{i=1}^{n-k+1} j_i = k$ and $\sum_{i=1}^{n-k+1} i j_i = n$.
Using this identity and the fact that $M^{(k)}(0)=k! \mu_k$, we have
\begin{equation*} 
n! u_n = U^{(n)}(0) 
= \sum_{k=1}^n (-1)^k k! 
        B_{n,k}\left( -\mu_1, -2\mu_2, \dots,
                -(n-k+1)! \mu_{n-k+1} \right).
\end{equation*}
A basic property of Bell polynomials \citep[][page 412]{Charalambides2002} states that
\begin{equation} \label{eq:belliden}
B_{n,k}\left(abx_1,a^2bx_2,\dots,a^{n-k+1} b x_{n-k+1} \right)
= a^n b^k B_{n.k}\left( x_1, x_2,\dots, x_{n-k+1} \right).
\end{equation}
Using this identity with $a=1$ and $b=-1$, it follows that
\begin{equation*} 
n! u_n 
= \sum_{k=1}^n k! B_{n,k}\left( \mu_1, 2\mu_2, \dots, (n-k+1)! \mu_{n-k+1} \right),
\end{equation*}
and we have proved the following theorem.

\begin{theorem} \label{thm:PEn}
Let $\bmu$ be a probability distribution on the positive integers.
Then
\begin{equation*}
\Pr[ E_n \mid \bmu ]
= \sum_{k=1}^n \frac{ k! }{n!}
        B_{n,k}\left( \mu_1, 2\mu_2, \dots, (n-k+1)! \mu_{n-k+1} \right).
\end{equation*}
\end{theorem}

\paragraph{Example: ESC-Poisson.}
Consider the case in which the sequence $( \mu_n )_{n=0}^\infty$ is given by
\begin{equation*}
\mu_k 
= \begin{cases}
\frac{ \lambda^{k-1} e^{-\lambda} }{ (k-1)! }
= \frac{ k e^{-\lambda} }{ \lambda } \frac{ \lambda^k }{ k! }
        &\mbox{ if } k = 1,2,\dots \\
0 &\mbox{ if } k=0.
\end{cases}
\end{equation*}
That is, cluster sizes are shifted Poisson random variables.
Applying Theorem~\ref{thm:PEn},
\begin{equation*} \begin{aligned}
\Pr[E_n \mid \lambda ]
&= \sum_{k=1}^n \frac{ k! }{ n! }
  B_{n,k}\left( \frac{ e^{-\lambda} }{\lambda} \lambda,
                2 \frac{ e^{-\lambda} }{\lambda} \lambda^2, \dots,
                (n-k+1) \frac{ e^{-\lambda} }{\lambda} \lambda^{n-k+1}
        \right) \\
&= \sum_{k=1}^n \frac{ k! e^{-k\lambda} \lambda^{n-k} }{ n ! }
                B_{n,k}\left( 1,2,\dots,(n-k+1) \right),
\end{aligned} \end{equation*}
where we have used the property in Equation~\eqref{eq:belliden}.
A basic Bell polynomial identity~\citep[][page 135]{Comtet1974} states that
\begin{equation} \label{eq:bell:idempotent}
  B_{n,k}\left(1,2,\dots,(n-k+1) \right)
  = \binom{n}{k} k^{n-k}.
\end{equation}
Applying this identity, we conclude that
\begin{equation} \label{eq:pois:un}
\Pr[ E_n \mid \lambda ]
= \sum_{k=1}^n \frac{ e^{-k\lambda } (k \lambda)^{n-k} }{ (n-k)! }
= \sum_{k=1}^n \Pois( n-k ; k\lambda ),
\end{equation}
where $\Pois( \cdot ; \lambda )$ denotes the probability mass function of a Poisson random variable with rate parameter $\lambda$.
Appendix~\ref{apx:examples} includes similar computations for other cluster size distributions.

\subsection{Behavior of the number of clusters $K_n$}
\label{subsec:theory:K}

The number of clusters $K_n$ is the (random) number $k$ such that
$\sum_{j=1}^k S_j = n$, again conditional on the event $E_n$ to ensure that such a $k$ exists.
We begin by observing that
\begin{equation*}
E_n = \cup_{k=1}^n \left\{ K_n = k \right\},
\end{equation*}
whence for $k=1,2,\dots,n$,
\begin{equation*} \begin{aligned}
\Pr[ K_n = k \mid E_n, \bmu ]
&= \frac{1}{\Pr[ E_n \mid \bmu ] }\sum_{s_1,s_2,\dots,s_k}
	\Pr[ S_1=n_1, S_2=s_2, \dots, S_k=s_k \mid E_n, \bmu ] \\
&= \frac{1}{\Pr[ E_n \mid \bmu ] } \sum_{s_1,s_2,\dots,s_k}
	\prod_{j=1}^k \bmu_{s_j},
\end{aligned} \end{equation*}
where the sum is over all $s_1,s_2,\dots,s_k$ satisfying $\sum_{j=1}^k s_j = n$.
Equivalently, using basic properties of partitions of $[n]$, we can express this sum as
\begin{equation} \label{eq:Kn:permute}
\Pr[ K_n = k \mid E_n, \bmu ]
= \frac{1}{\Pr[ E_n \mid \bmu ] } \sum_{ j_1,j_2,\dots,j_{n-k+1} }
	\frac{ k! }{ j_1! j_2 ! \cdots j_{n-k+1}! }
        	\prod_{i=1}^k \bmu_{j_i},
\end{equation}
where now the sum is over all $j_1,j_2,\dots,j_{n-k+1}$ satisfying $\sum_{i=1}^{n-k+1} j_i = k$ and $\sum_{i=1}^{n-k+1} i j_i = n$.
The sum on the right-hand side of Equation~\eqref{eq:Kn:permute} is known in the enumerative combinatorics literature as the ordinary Bell polynomial~\citep[][]{Charalambides2002},
\begin{equation} \label{eq:def:OrdinaryBell}
\Bhat_{n,k}( \mu_1, \mu_2, \dots, \mu_{n-k+1})
= \sum_{j_1,j_2,\dots,j_{n-k+1}} \frac{ k! }{ j_1! j_2 ! \cdots j_{n-k+1}! }
        \prod_{i=1}^k \bmu_{j_i},
\end{equation}
and can be related to the exponential Bell polynomial defined in Equation~\eqref{eq:def:ExponentialBell} according to
\begin{equation*} 
\Bhat_{n,k}( \mu_1, \mu_2, \dots, \mu_{n-k+1})
= \frac{k!}{n!}
B_{n,k}\left( 1! \mu_1, 2! \mu_2, \dots, (n-k+1)! \mu_{n-k+1} \right).
\end{equation*}
Thus, we have proved the following result.

\begin{theorem} \label{thm:Kprob}
Let $S_1,S_2,\dots,S_{K_n}$ be cluster sizes generated according to an ESC model on $n$ objects with cluster size distribution $\bmu$.
Then for $k=1,2,\dots,n$,
\begin{equation} \label{eq:probKn}
\Pr[ K_n = k \mid E_n, \bmu ] =
\frac{ \Bhat_{n,k}( \bmu )}{\Pr[ E_n \mid \bmu ] }
= \frac{k!
	B_{n,k}\left( 1! \mu_1, 2! \mu_2, \dots, (n-k+1)! \mu_{n-k+1} \right) }
	{n! \Pr[ E_n \mid \bmu ]}.
\end{equation}
\end{theorem}

With this result in hand, provided we can evaluate Bell polynomials on the sequence $\bmu$, we can precisely describe the behavior of $K_n$ for a particular choice of $\bmu$ (or a prior over $\bmu$).

\paragraph{Example: Negative Binomial Cluster Sizes.}
By way of illustration, we consider the model that has received the most attention to date in the microclustering literature \citep[see, e.g.,][]{ZanellaETALneurips2016,BetZanSte2020,NatIorHeiMayGle2021}, the ESC-NB model.
Under this model, $\bmu$ takes the form of a shifted negative binomial distribution,
\begin{equation*}
\mu_k = \begin{cases}
        \binom{ k + r - 2 }{ k-1 } (1-p)^r p^{k-1}
        &\mbox{ if } k=1,2,\dots \\
        0 &\mbox{ if } k = 0,
        \end{cases}
\end{equation*}
where $p \in [0,1]$ is the probability of success and $r > 0$ is the number of failures.
To permit the possibility that $r > 0$ is not an integer, we define
\begin{equation*}
\binom{ r }{ m } = \frac{ (r)_{m} }{ m ! },
\end{equation*}
where $(r)_m$ denotes the falling factorial, $(r)_m = r(r-1)(r-2)\cdots(r-m+1)$.
Using binomial identities and properties of the Bell polynomials, we find that under the ESC-NB model,
\begin{equation} \label{eq:K:nb}
\Pr[ K_n = k \mid E_n, \bmu ]
=
\frac{ p^{n-k} (1-p)^{rk} }{ \Pr[ E_n \mid \bmu ] }
	\binom{ n +k(r-1) -1 }{ n - k },
\end{equation}
where $\Pr[ E_n \mid \bmu]$ is given by Theorem~\ref{thm:PEn}.
Thus, Theorem~\ref{thm:Kprob} applied to the ESC-NB model
recovers Proposition 2 in \cite{NatIorHeiMayGle2021} as a special case.
See Appendix~\ref{apx:examples} for details of this computation, including a closed form for $\Pr[ E_n \mid \bmu]$, and additional examples.

%% file: experiments.tex
We now turn to a brief experimental investigation of our theoretical results.

\subsection{Behavior of $K_n$}
We begin by verifying that the samples generated by Algorithm~\ref{alg:main}
match their intended ESC clustering distribution
(i.e., verifying Theorem~\ref{thm:algmain:correctness}).
Theorem~\ref{thm:Kprob} establishes the distribution of the number of clusters $K_n$ under ESC models.
In particular, Equation~\eqref{eq:K:nb} gives the distribution of $K_n$ under the ESC-NB model, in which the cluster sizes are distributed according to a Negative Binomial with success parameter $p \in [0,1]$ and number of failures $r > 0$.
Figure~\ref{fig:nb_Khist} shows this distribution for $p=0.5$ and $r=2.0$.
The left-hand plot contains a histogram of 2000 draws of $K_n$, based on clusterings generated from the na\"{i}ve ESC sampling method \citep{BetZanSte2020}.
The right-hand plot contains an analogous histogram based on clusterings generated from Algorithm~\ref{alg:main}, computing the $\Pr[ E_n \mid \bmu]$ terms using Bell polynomial identities.
In both subplots, the black line indicates the distribution of $K_n$ predicted by Theorem~\ref{thm:Kprob}.
We see that both the na\"{i}ve and Bell polynomial-based algorithms yield clusterings in which the behavior of $K_n$ matches that predicted by Theorem~\ref{thm:Kprob}.

\begin{figure}[ht]
  \centering
  \includegraphics[width=0.8\columnwidth]{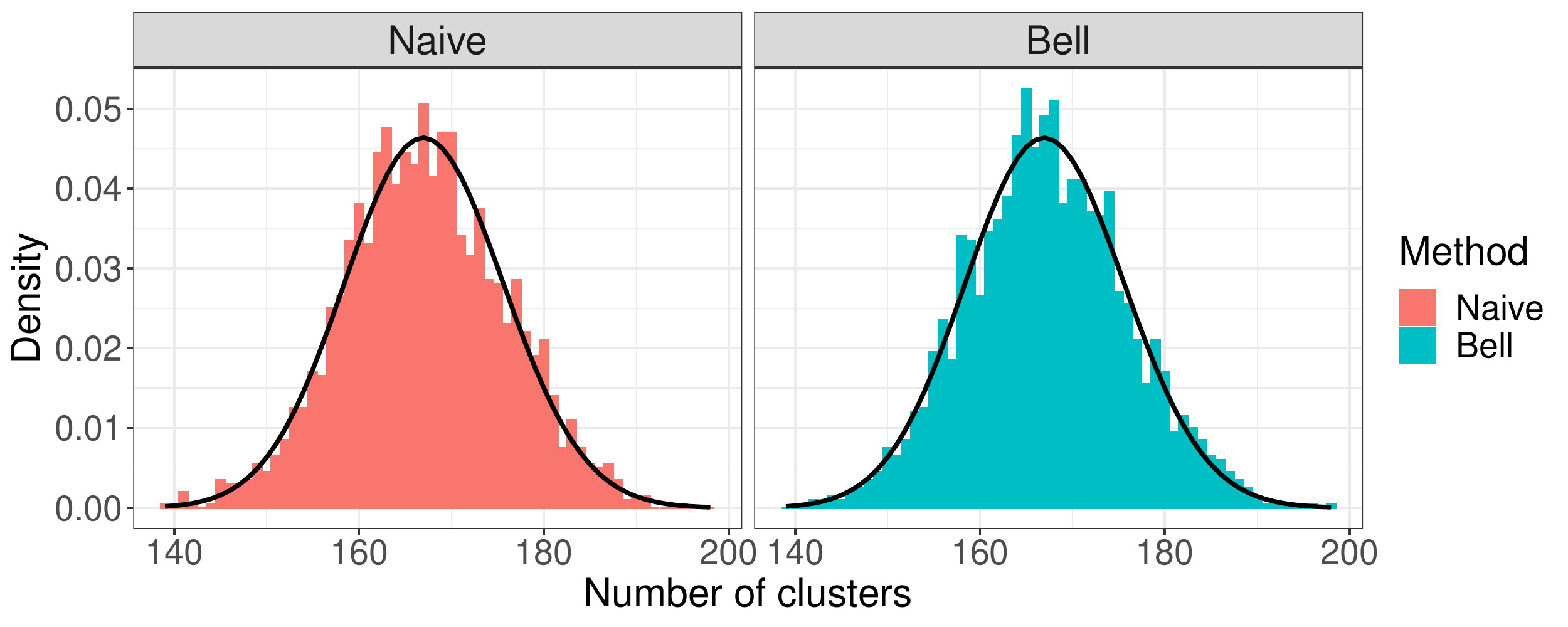}
  \vspace{-1mm}
  \caption{Histogram of $2000$ draws from the distribution of the number of clusters $K_n$ under the ESC-NB model on $n=500$ observations with Negative Binomial parameters $p=0.5$ and $r=2.0$ using the na\"{i}ve (left) and Bell polynomial-based (right) sampling algorithms.
The distribution predicted by Theorem~\ref{thm:Kprob} is indicated in black in both subplots.}
  \label{fig:nb_Khist}
\end{figure}

\subsection{Runtime Comparison}
We now turn to a comparison of our proposed sampling algorithm with the
na\"{i}ve sampling approach described in 
Section~\ref{subsec:theory:gen} and used in most previous microclustering
work~\citep[see, e.g.,][]{BetZanSte2020}.
For simplicity, we consider the ESC-Poisson model,
in which cluster sizes are drawn according to a Poisson distribution
with parameter $\lambda$.

Lemma~\ref{lem:renewal} suggests that the runtime of the
na\"{i}ve sampling algorithm is likely to be
sensitive to the mean of the cluster size distribution $\E S_1 = \lambda$.
To examine this fact, we generated partitions of $n=500$ objects
under the ESC-Poisson model with Poisson parameter $\lambda$
using both the na\"{i}ve procedure
and the procedure described in Algorithm~\ref{alg:main}.
For varying values of the Poisson parameter $\lambda$,
we performed $20$ independent repetitions, recording the runtime
required to generate clusterings under both methods.
The mean runtime over these $20$ replicates for these two methods are
summarized in Figure~\ref{fig:poistiming_fixedn_varylambda},
with the na\"ive method indicated by orange circles
and the Bell polynomial-based method indicated by teal triangles.
We see that the runtime of the na\"{i}ve sampling error depends sensitively on
the mean $\lambda$ of the cluster size distribution.
Specifically, runtimes for the na\"ive method
are orders of magnitude slower for values of $\lambda$
that do not (exactly or approximately) divide $n=500$.
Under such circumstances, if $S_1,S_2,\dots,S_k$ are such that
$\sum_{j=1}^k S_j = n$,
either all of the summands must be moderately far
from the mean $\E S_1 = \lambda$ of the cluster size distribution,
or, if most of the summands are close to $\E S_1$,
one or more must deviate significantly from it.
In either event, such sequences are of especially low probability,
and thus many sequences $S_1,S_2,\dots$ must be generated before the
event $E_n$ occurs, increasing the average runtime of the na\"{i}ve procedure.

\begin{figure}[hb]
  \centering
  \includegraphics[width=0.8\columnwidth]{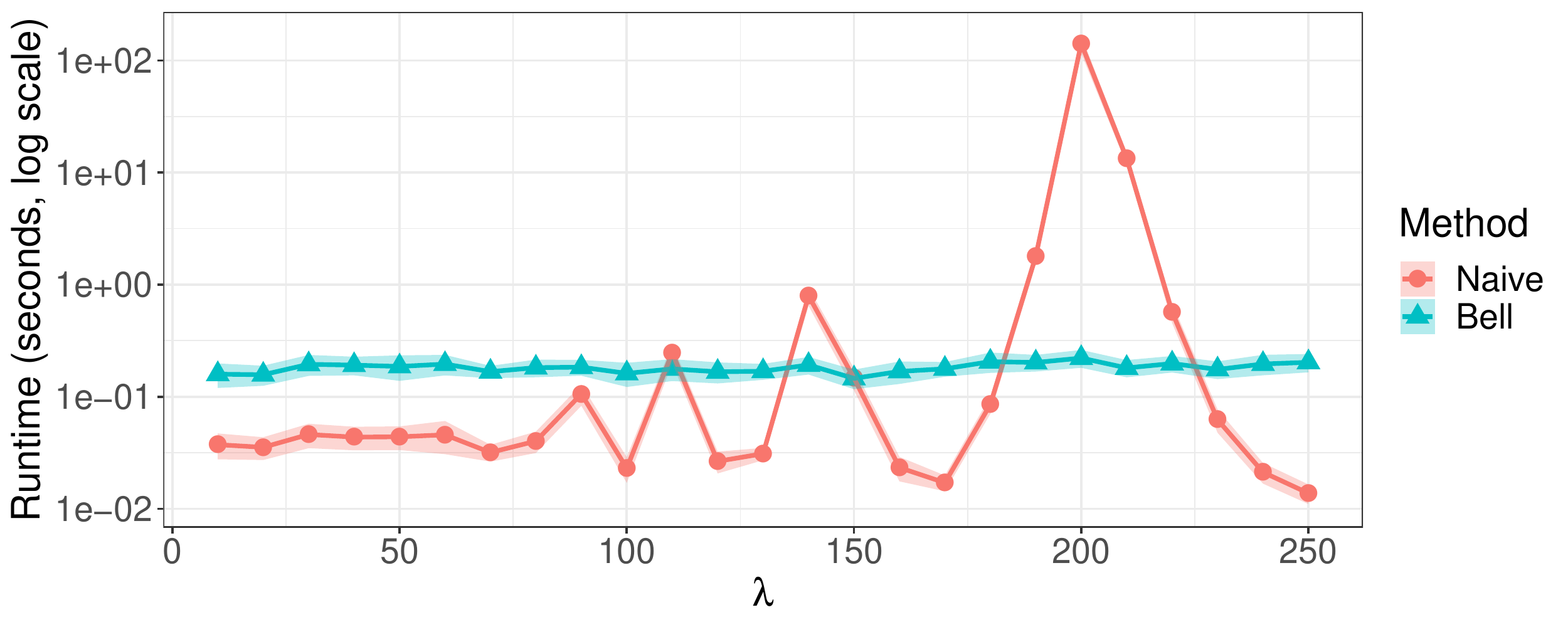}
  \vspace{-1mm}
  \caption{Runtime in seconds required by the na\"{i}ve (orange circles) and Bell polynomial-based (teal triangles) methods to generate a partition of $n=500$ objects under an ESC-Poisson model, as a function of the Poisson parameter $\lambda$. We see that the na\"ive sampling method is highly sensitive to the expected cluster size $\E S_1 = \lambda$.}
  \label{fig:poistiming_fixedn_varylambda}
\end{figure}

Further examining Figure~\ref{fig:poistiming_fixedn_varylambda},
we note that our proposed sampling method does not uniformly improve upon
the na\"ive sampling method at all values of $\lambda$.
This is owing to the fact that Algorithm~\ref{alg:main} requires
that we compute the probabilities
\begin{equation} \label{eq:mainalgo:prob}
\Pr[ X_k = s; m ]
= \frac{ \bmu_{s} \Pr[E_{m-s} \mid \bmu] }{ \Pr[ E_m \mid \bmu ] }
\end{equation}
for each $m \in [n]$ and each $s \in [m]$.
Even with access to the sequences $\bmu_m$ and $u_m = \Pr[ E_m \mid \bmu ]$
for $m \in [n]$,
constructing these probabilities ahead of time incurs a computational cost,
which is included in the runtime reported
in Figure~\ref{fig:poistiming_fixedn_varylambda}.

Figure~\ref{fig:poistiming_fixedlambda_varyn} compares
the na\"ive sampling procedure and our proposed method,
this time amortizing this up-front computational cost over $200$ sample partitions.
That is, each trial now consists of first calculating
the probabilities in Equation~\eqref{eq:mainalgo:prob},
then using those probabilities to generate $200$ clusterings
from the ESC-Poisson model.
We see that over a range of values of Poisson parameter $\lambda$
and number of observations $n$,
our proposed method improves upon the runtime of the
na\"ive sampling method by an order of magnitude.

\begin{figure}[ht]
  \centering
  \includegraphics[width=0.8\columnwidth]{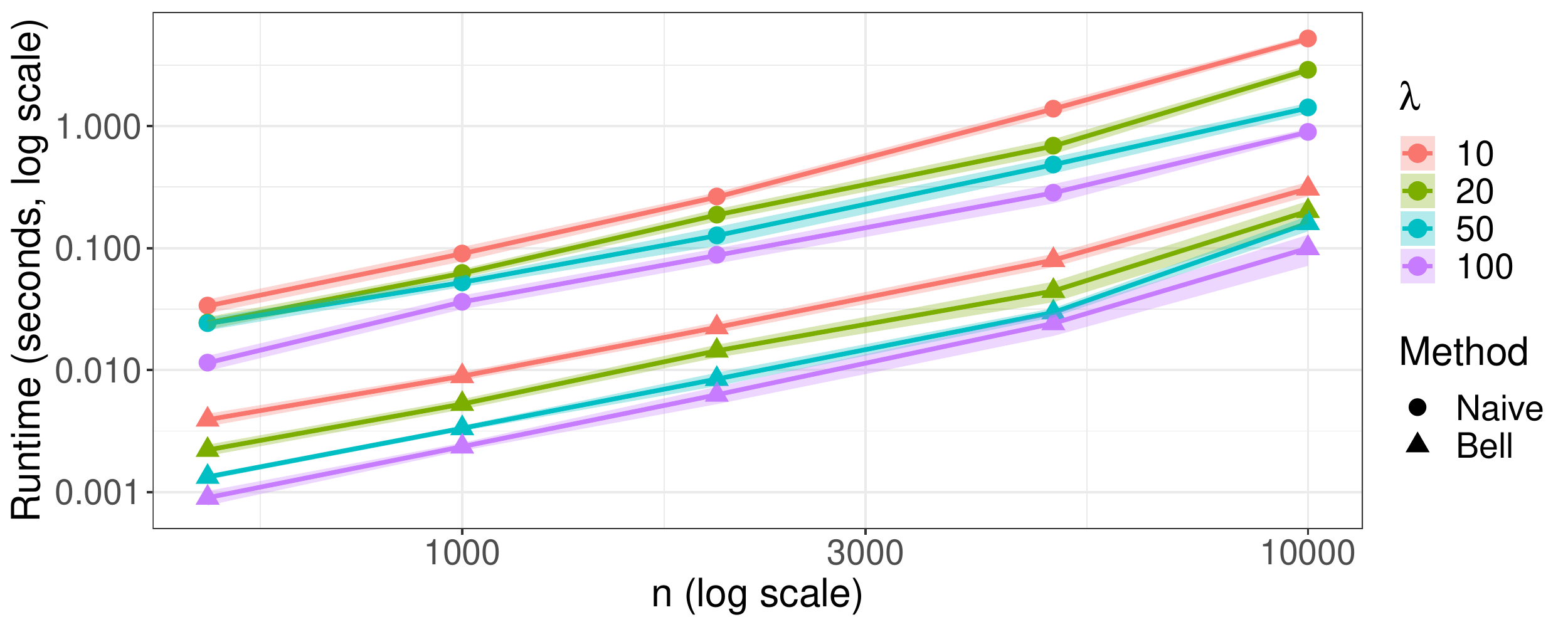}
  \vspace{-1mm}
  \caption{Amortized runtime of the na\"{i}ve ESC sampler (circles) and our proposed Bell polynomial-based method (triangles), as a function of the number of observations $n$. Cluster sizes were generated according to a Poisson distribution with varying choices of mean $\lambda$ (indicated by line color). Each point corresponds to the mean runtime over $20$ trials, with error bars indicating two standard errors of the mean. In each trial, $200$ samples were generated from the ESC-Poisson model with parameter $\lambda$, and total runtime, including up-front computation required by the Bell polynomial-based method, was recorded.}
  \label{fig:poistiming_fixedlambda_varyn}
\end{figure}

%% file: apx_algmain.tex
Here, we give a detailed proof of the correctness of Algorithm~\ref{alg:main} as claimed by Theorem~\ref{thm:algmain:correctness}.

\begin{proof}
By construction of Algorithm~\ref{alg:main}, the variable $m$ is initialized to $m \gets n$, and thus
\begin{equation*}
\Pr[ X_1=s_1 \mid \bmu ]
=
\frac{ \bmu_{s_1} \Pr[E_{n-s_1} \mid \bmu] }{ \Pr[ E_n \mid \bmu ] }.
\end{equation*}
It follows that
\begin{equation*} \begin{aligned}
\Pr\left[ X_1=s_1,\dots,X_k=s_k \mid \bmu \right]
&= \Pr[ X_1=s_1 \mid \bmu ]
        \Pr\left[ X_2=s_2,\dots,X_k=s_k \mid X_1=s_1, \bmu \right] \\
&= \frac{ \bmu_{s_1} \Pr[E_{n-s_1} \mid \bmu] }{ \Pr[ E_n \mid \bmu ] }
        \Pr\left[ X_2=s_2,\dots,X_k=s_k \mid X_1=s_1, \bmu \right]
\end{aligned} \end{equation*}
After drawing $X_1=s_1$, Algorithm~\ref{alg:main} sets $m \gets n-s_1$ and $k \gets 2$, and draws $X_2$ according to
\begin{equation*}
\Pr\left[ X_2=s_2 \mid X_1=s_1, \bmu \right]
= \frac{ \bmu_{s_2} \Pr[E_{n-s_1-s_2} \mid \bmu] }
        { \Pr[ E_{n-s_1} \mid \bmu ] },
\end{equation*}
whence
\begin{equation*} \begin{aligned} 
\Pr\left[ X_1=s_1,\dots,X_k=s_k \mid \bmu \right]
&=
\frac{ \bmu_{s_1} \Pr[E_{n-s_1} \mid \bmu ] }{ \Pr[ E_n \mid \bmu ] }
        \cdot
        \frac{ \bmu_{s_2} \Pr[E_{n-s_1-s_2} \mid \bmu] }
                { \Pr[ E_{n-s_1} \mid \bmu ] } \\
&~~~~~~\cdot
        \Pr\left[ X_3=s_3,\dots,X_k=s_k
                \mid (X_1,X_2)=(s_1,s_2), \bmu \right].
\end{aligned} \end{equation*}
Repeating this argument, we have
\begin{equation*} \begin{aligned}
&\Pr\left[ X_1=s_1,\dots,X_k=s_k \mid \bmu \right] \\
&~= \left( \prod_{j=1}^{k-1}
        \frac{ \bmu_{s_j} \Pr[E_{n-\sum_{t=1}^j s_t} \mid \bmu] }
                { \Pr[ E_{n-\sum_{t=1}^{j-1} s_t} \mid \bmu ] } \right)
        \Pr[ X_k=s_k \mid X_1=s_1,\dots,X_{k-1}=s_{k-1},\bmu ] \\
&~= \left( \prod_{j=1}^{k-1}
        \frac{ \bmu_{s_j} \Pr[E_{n-\sum_{t=1}^j s_t} \mid \bmu] }
                { \Pr[ E_{n-\sum_{t=1}^{j-1} s_t} \mid \bmu ] } \right)
        \frac{ \bmu_{s_k} \Pr[ E_0 \mid \bmu ] }{ \Pr[ E_{n-\sum_{t=1}^{k-1} s_t} \mid \bmu ] }.
\end{aligned} \end{equation*}
Using the fact that $\Pr[ E_0 \mid \bmu ]=1$,
we conclude that
\begin{equation*} 
\Pr\left[ X_1=s_1,\dots,X_k=s_k \mid \bmu \right]
= \frac{ \left( \prod_{j=1}^k  \bmu_{s_j} \right) }{ \Pr[E_n \mid \bmu ] }.
\end{equation*}
On the other hand, repeated application of Equation~\eqref{eq:seqstep} yields
\begin{equation*}
\Pr[ S_1=s_1,S_2=s_2,\dots,S_k=s_k \mid E_n, \bmu ]
=
\frac{ \left( \prod_{j=1}^k \bmu_{s_j} \right) }{ \Pr[ E_n \mid \bmu ] }.
\end{equation*}
Comparison of the above two displays yields the result.
\end{proof}

%% file: apx_examples.tex
In this section, we provide illustrative computations for several natural choices of cluster size distriubtions.

\subsection{Poisson Cluster Sizes}
\label{apx:examples:Poisson}

Under the ESC-Poisson distribution, as introduced in Section~\ref{subsec:theory:gen}, cluster sizes are drawn according to a (shifted) Poisson,
\begin{equation*}
\mu_k 
= \begin{cases}
\frac{ \lambda^{k-1} e^{-\lambda} }{ (k-1)! }
= \frac{ k e^{-\lambda} }{ \lambda } \frac{ \lambda^k }{ k! }
        &\mbox{ if } k = 1,2,\dots \\
0 &\mbox{ if } k=0.
\end{cases}
\end{equation*}

In Section~\ref{subsec:theory:gen}, we determined the form of $\Pr[E_n \mid \bmu]$ for this distribution.
Here, we use Theorem~\ref{thm:Kprob} to derive the distribution of the number of clusters $K_n$.
By Theorem~\ref{thm:Kprob}, for $k \in [n]$,
\begin{equation*} \begin{aligned}
\Pr[ K_n = k \mid E_n, \bmu ] &= \frac{ \Bhat_{n,k}( \bmu ) }{ \Pr[ E_n \mid \bmu ] }
=
\frac{ k! }{n!} B_{n,k}\left(\mu_1,2\mu_2,\dots,(n-k+1)! \mu_{n-k+1} \right) \\
&= \frac{k!
  B_{n,k}\left( \frac{ e^{-\lambda} }{\lambda} \lambda,
                2 \frac{ e^{-\lambda} }{\lambda} \lambda^2, \dots,
                (n-k+1) \frac{ e^{-\lambda} }{\lambda} \lambda^{n-k+1}
        \right) }
{n! \Pr[ E_n \mid \bmu ] } \\
&=
\left( \sum_{\ell=1}^n \frac{ e^{-\ell\lambda } (\ell \lambda)^{n-\ell} }
	{ (n-\ell)! } \right)^{-1}
\frac{ e^{-k\lambda } (k \lambda)^{n-k} }{ (n-k)! }
\end{aligned} \end{equation*}
where we have used Equations~\eqref{eq:belliden},~\eqref{eq:bell:idempotent} and~\eqref{eq:pois:un} as in the example given in Section~\ref{subsec:theory:gen}.

\subsection{Negative Binomial Cluster Sizes}
\label{apx:examples:NB}

Recall that under the ESC-NB distribution, as introduced in Section~\ref{subsec:theory:K}, the cluster sizes are drawn i.i.d.\ according to a (shifted) negative binomial,
\begin{equation*}
\mu_k = \begin{cases}
	\binom{ k + r - 2 }{ k-1 } (1-p)^r p^{k-1}
	&\mbox{ if } k=1,2,\dots \\
	0 &\mbox{ if } k = 0
	\end{cases}
\end{equation*}
where $r > 0$ and $p \in [0,1]$, and we recall that
\begin{equation*}
\binom{ r }{ m } = \frac{ ( r )_m }{ m ! },
\end{equation*}
where $(r)_m$ denotes the falling factorial.
That is, if $Z$ is a Negative Binomial random variable
with success parameter $p \in [0,1]$ and number of failures $r > 0$,
then $\mu_k = \Pr[ Z = k-1 ]$.

Section~\ref{subsec:theory:K} gives the distribution of the number of clusters $K_n$ under this cluster size distribution, up to the normalizing constant $\Pr[E_n \mid \bmu]$.
Here, we establish a closed-form expression for this normalizing term, using the tools introduced in Section~\ref{subsec:theory:gen}.
By Theorem~\ref{thm:PEn},
\begin{equation} \label{eq:NB:un}
\Pr[ E_n \mid \bmu]
= \sum_{k=1}^n \Bhat_{n,k}( \mu_1,\mu_2,\dots,\mu_{n-k+1} ).
\end{equation}
By definition of the ordinary Bell polynomials given in Equation~\eqref{eq:def:OrdinaryBell},
\begin{equation*} 
\Bhat_{n,k}( \mu_1,\mu_2,\dots,\mu_{n-k+1} )
= \sum_{j_1,j_2,\dots,j_{n-k+1}} \frac{k!}{j_1!j_2!\cdots j_{n-k+1}!}
		\prod_{i=1}^{n-k+1} \mu_i^{j_i} \\
= \sum_{s_1,s_2,\dots,s_k}
	\prod_{i=1}^k \mu_{s_k},
\end{equation*}
where the second sum is over all positive integers $s_1,s_2,\dots,s_k$
summing to $n$.
Plugging in our definitions for $\bmu$ under the negative binomial, this becomes
\begin{equation*} \begin{aligned}
\Bhat_{n,k}( \mu_1,\mu_2,\dots,\mu_{n-k+1} )
&= \sum_{s_1,s_2,\dots,s_k} \prod_{i=1}^k
	\binom{ s_i + r -2 }{ s_i-1 } (1-p)^r p^{s_i-1} \\
&= (1-p)^{rk} p^{n-k}
	\sum_{s_1,s_2,\dots,s_k} \prod_{i=1}^k
	\binom{ s_i + r -2 }{ s_i-1 },
\end{aligned} \end{equation*}
where again all sums are over positive integers $s_1,s_2,\dots,s_k$ summing to $n$.
After a change of variables, we have
\begin{equation*}
\Bhat_{n,k}( \mu_1,\mu_2,\dots,\mu_{n-k+1} )
= (1-p)^{rk} p^{n-k}
	\sum_{s_1,s_2,\dots,s_k} \prod_{i=1}^k
	\binom{ s_i + r -1 }{ s_i },
\end{equation*}
where now the sum is over all non-negative integers $s_1,s_2,\dots,s_k$ summing to $n-k$.
A basic identity for binomial coefficients \citep[][Equation 5.14]{Concrete} states that
\begin{equation} \label{eq:concrete:negation}
\binom{t}{m} = (-1)^m \binom{m-t-1}{m},
\end{equation}
which holds for all $t \in \R$ and non-negative integer $m$.
Taking $t = s_i+r-2$ and $m=s_i-1$,
\begin{equation*} \begin{aligned}
\Bhat_{n,k}( \mu_1,\mu_2,\dots,\mu_{n-k+1} )
&= \frac{ (1-p)^{rk} p^n }{ p^k }
        \sum_{s_1,s_2,\dots,s_k} \prod_{i=1}^k
	(-1)^{s_i} \binom{-r}{s_i} \\
&= \frac{ (-1)^{n-k} (1-p)^{rk} p^n }{ p^k }
        \sum_{s_1,s_2,\dots,s_k} \prod_{i=1}^k
	\binom{ -r }{ s_i }.
\end{aligned} \end{equation*}
Applying the generalized Vandermonde convolution
identity~\citep{Concrete},
a second application of Equation~\eqref{eq:concrete:negation} yields
\begin{equation*} \begin{aligned}
\Bhat_{n,k}( \mu_1,\mu_2,\dots,\mu_{n-k+1} )
&= (-1)^{n-k} (1-p)^{rk} p^{n-k} \binom{ -kr }{ n-k } \\
&= (-1)^{n-k} (1-p)^{rk} p^{n-k} (-1)^{n-k} \binom{ n-k + kr -1 }{ n - k } \\
&= p^n \left( \frac{ (1-p)^r }{ p } \right)^k \binom{ n +k(r-1) -1 }{ n - k }
\end{aligned} \end{equation*}
Plugging this back into Equation~\eqref{eq:NB:un},
\begin{equation*} 
\Pr[ E_n \mid \bmu ]
=
p^n \sum_{k=1}^n \left( \frac{ (1-p)^r }{ p } \right)^{k}
	\binom{ n +k(r-1) -1 }{ n - k }.
\end{equation*}

\subsection{Geometric Cluster Sizes}

As another illustrative example, consider the setting where
cluster sizes are distributed according to a geometric distribution,
\begin{equation*}
\mu_k = \begin{cases}
	(1-p)^{k-1} p &\mbox{ if } k=1,2,\dots \\
	0 &\mbox{ if } k = 0.
	\end{cases}
\end{equation*}
Applying Theorem~\ref{thm:PEn}, we obtain
\begin{equation*} \begin{aligned}
\Pr[ E_n \mid \bmu ]
&= \sum_{k=1}^n \frac{ k! }{ n! }
	 B_{n,k}\left( p, 2(1-p)p, 3!(1-p)^2 p,
			\dots,(n-k+1)!(1-p)^{n-k} p \right) \\
&= \sum_{k=1}^n \frac{ k! }{ n! }
		\left( \frac{ p }{1-p} \right)^k (1-p)^n
		B_{n,k}\left(1!,2!,3!,\dots,(n-k+1)! \right),
\end{aligned} \end{equation*}
where we have again used the identity in Equation~\eqref{eq:belliden}.
A basic identity \citep[][page 135]{Comtet1974} states that
\begin{equation} \label{eq:bell:lah}
  B_{n,k}\left(1!,2!,3!,\dots,(n-k+1)! \right)
  = \binom{n-1}{k-1} \frac{ n! }{ k! },
\end{equation}
from which we conclude that, after a change of variables,
\begin{equation*}
\Pr[ E_n \mid \bmu ]
= (1-p)^n \sum_{k=1}^n \binom{n-1}{k-1} 
			\left( \frac{ p }{1-p} \right)^k 
= (1-p)^{n-1} p
  \sum_{\ell=0}^{n-1} \binom{n-1}{\ell} \left( \frac{ p }{1-p} \right)^\ell
= p.
\end{equation*}

Turning to the cluster size distribution under this model, Theorem~\ref{thm:Kprob} states that
\begin{equation*}
\Pr[ K_n = k \mid E_n, \bmu ]
=
\frac{ \Bhat_{n,k}( \bmu ) }{ \Pr[ E_n \mid \bmu ] }
=
\frac{ k! }{ n! p } 
B_{n,k}( \mu_1, 2 \mu_2, \dots, (n-k+1)! \mu_{n-k+1} ).
\end{equation*}
Applying identities~\eqref{eq:belliden} and~\eqref{eq:bell:lah} yields
\begin{equation*}
\Pr[ K_n = k \mid E_n, \bmu ]
= \frac{1}{p} \left( \frac{ p }{1-p} \right)^k (1-p)^n \binom{n-1}{k-1}.
\end{equation*}

\subsection{ESC-Zipf}
\label{apx:subsec:zipf}

Consider, for $\alpha > 1$, a Zipfian cluster size distribution, given by
\begin{equation*}
\mu_k = \frac{ k^{-\alpha} }{ \zeta( \alpha ) } \text{ for } k=0,1,2,\dots,
\end{equation*}
where $\zeta( \cdot )$ denotes the Riemann zeta function.
Then our results above imply that
\begin{equation*} \begin{aligned}
u_n
&= \sum_{k=1}^n \frac{ k! }{ n! }
        B_{n,k}\left( \frac{1}{\zeta( \alpha )},
			\frac{2!~ 2^{-\alpha}}{ \zeta(\alpha)},
			\frac{3!~ 3^{-\alpha}}{ \zeta(\alpha)},
			\dots,
			\frac{ (n-k+1)! ~(n-k+1)^{-\alpha} }{ \zeta(\alpha) }
		\right) \\
&= \sum_{k=1}^n \frac{ k! }{ n! \zeta^k(\alpha) }
	B_{n,k}\Big( 1, 2! ~ 2^{-\alpha}, 3! ~ 3^{-\alpha}, \dots,
			(n-k+1)! ~(n-k+1)^{-\alpha} \Big).
\end{aligned} \end{equation*}
It is not immediately clear how to simplify this probability using basic Bell polynomial identities.
Nonetheless, from Theorem~\ref{thm:Kprob}, we have that for $k \in [n]$,
\begin{equation*}
\Pr[ K_n = k \mid E_n, \alpha ]
= \frac{ k! \zeta^{-k}(\alpha) B_{n,k}\left( 1,
                        2! ~ 2^{-\alpha},
                        \dots,
                        (n-k+1)!~(n-k+1)^{-\alpha} \right) }
	{\sum_{\ell=1}^n \ell! \zeta^{-\ell}(\alpha)
		B_{n,\ell}\left( 1, 2! ~ 2^{-\alpha}, \dots,
                        (n-\ell+1)!~ (n-\ell+1)^{-\alpha} \right) },
\end{equation*}
and the Bell polynomials appearing on the right-hand side can be computed in quadratic time according to the recurrence relation~\citep[][Equations 11.11, 11.12]{Charalambides2002}
\begin{equation*}
B_{n,k}\left( \mu_1, \mu_2,\dots,\mu_{n-k+1} \right)
= \sum_{j=1}^{n-k+1} \mu_j
	B_{n-j,k-1}\left( \mu_1, \mu_2,\dots,\mu_{n-j-k} \right).
\end{equation*}
Thus, even in the absence of a closed-form expression for $\Pr[ K_n \mid E_n, \alpha ]$, the distribution of $K_n$ can be obtained numerically.